\documentclass[10pt]{article}

 \usepackage{amsfonts,amssymb,graphicx,amsmath, amsthm}
\usepackage[colorlinks=true,linkcolor=blue,citecolor=blue]{hyperref}

\usepackage{multicol}

\newtheorem{rema}{Remark}

\newtheorem{lemm}{Lemma}
\newtheorem{theo}{Theorem}

\newcommand{\C}[1][]{\ensuremath{{\mathbb{C}^{#1}} }}
\newcommand{\R}[1][]{\ensuremath{{\mathbb{R}^{#1}} }}
\renewcommand{\S}[1][]{\ensuremath{{\mathbb{S}^{#1}} }}
\renewcommand{\H}[1][]{\ensuremath{{\mathbb{H}^{#1}} }}

\newcommand{\D}[1][]{\ensuremath{{\mathbb{D}^{#1}} }}
\renewcommand{\P}[1][]{\ensuremath{{\mathbb{P}^{#1}} }}

\def\Re{ \mathrm{Re}\, }

\newcommand{\M}{{\cal M}}

\newcommand{\s}{{\cal S}}
\newcommand{\<}{\langle}
\renewcommand{\>}{\rangle}

\newcommand{\eps}{\epsilon}
\newcommand{\te}{\theta}
\newcommand{\ka}{\kappa}
\newcommand{\la}{\lambda}
\newcommand{\be}{\beta}

\date{}

\newcommand{\co}{{\texttt{cos}\varepsilon}}
\newcommand{\si}{{\texttt{sin}\varepsilon}}
\newcommand{\ta}{{\texttt{tan}\varepsilon}}
\newcommand{\coe}{{\texttt{cot}\varepsilon}}
\newcommand{\Ho}{\mbox{Hor}}

\title{Hopf Hypersurfaces in pseudo-Riemannian complex and para-complex space forms}
\author{ Henri Anciaux\footnote{Universidade de S\~ao Paulo; supported by CNPq (PDE 211682/2013-6)},
Konstantina Panagiotidou\footnote{Faculty of Engeneering, Aristotle University of Thessaloniki, Greece}}

\begin{document}

\maketitle

\centerline{\textbf {\large{Abstract}}}

\bigskip

{\small The study of real hypersurfaces in pseudo-Riemannian complex space forms and para-complex space forms, which are the pseudo-Riemannian generalizations of the complex space forms, is addressed. It is proved that there are no umbilic hypersurfaces, nor real hypersurfaces  with parallel shape operator in such spaces. Denoting by $J$ be the complex or  para-complex structure of a pseudo-complex or para-complex space form respectively, a  non-degenerate hypersurface of such space with unit normal vector field $N$ is said to be \em Hopf \em if the tangent vector field $JN$ is a principal direction. It is proved that if a hypersurface is Hopf, then the corresponding principal curvature (the \em Hopf \em curvature) is constant. It is also observed that
  in some cases a Hopf hypersurface must be, locally, a tube over a complex (or para-complex) submanifold, thus generalizing  previous results of Cecil, Ryan and Montiel.}

\bigskip

\centerline{\small \em 2010 MSC:  53C42, 53C40, 53B25
\em }


\section*{Introduction}

The study of  real hypersurfaces in  complex space forms, i.e.\ the complex projective space $\C\P^n$ and the complex hyperbolic space $\C\H^n$, have attracted a lot of attention in the last decades (see \cite{NR} for a survey of the subject and references therein). The complex structure $J$ of a complex space form induces a rich structure on real hypersurface; in particular, on an arbitrary oriented hypersurface $\s$ of $\C\P^n$ or $\C\H^n$ with unit vector normal field $N$, a canonical tangent field, called \em the structure vector field \em or \em the Reeb vector field, \em is defined by $\xi :=-JN$.  If $\xi$ is a principal direction on $\s$, i.e.\ an eigenvector of the shape operator, $\s$ is called a \em Hopf hypersurface\em. It turns out that
the principal curvature associated to the structure vector $\xi$ (the \em Hopf principal curvature\em) of a connected, Hopf hypersurface must be constant (this was proved in \cite{Ma} in the projective case and in \cite{KS} in the hyperbolic case).
 Moreover, in \cite{CR}, Hopf hypersurfaces in $\C\P^n$  are locally characterized as tubes over complex submanifolds, while in \cite{Mo}, the same statement is proved for Hopf hypersurfaces of $\C\H^n$ whose Hopf principal curvature  $a$  satisfies $|a|>2$. Recently Hopf hypersurfaces of  $\C\H^n$ with small Hopf principal curvature, i.e.\ satisfying $|a| \leq 2$, have been studied through a kind of generalized Gauss map in \cite{IR} and \cite{Iv}, while in \cite{Ki} a unified approach is proposed, relating Hopf hypersurfaces to totally complex (or para-complex) submanifolds of  some natural quaternionic manifold.

\medskip

The purpose of this paper is  to address the study of real hypersurfaces in \em pseudo-complex space forms \em $\C\P^n_p$, which are the pseudo-Riemannian generalizations of the complex space forms, and in \em para-complex space form \em $\D\P^n$. The latter space is the para-complex analog of $\C\P^n$ and is equipped with both a pseudo-Riemannian metric and  a \em para-complex \em structure, still denoted by $J$, which satisfies  $J^2=Id$. Furthermore, given a real hypersurface in $\D\P^n$ with non-degenerate induced metric, the Hopf field   is defined exactly as in the complex case. We refer to the next section for the precise  definition of $\D\P^n$ and a brief description of its geometry.
Since both the pseudo-complex and the para-complex case will be studied simultaneously, we define $\eps$ in such way that $J^2=-\eps Id$, i.e.\ $\eps=1$ corresponds to the complex case and $\eps=-1$ to the para-complex case.
Moreover, $\M$ will denote the  pseudo-Riemannian complex space form $\C\P^n_p$ or the para-complex space form $\D\P^n$, with holomorphic or para-holomorphic curvature $4c$, where $c: = \pm 1$.

\medskip
Our results are:

\begin{theo} \label{one} There exist no umbilic real hypersurface, nor  real hypersurface  with parallel shape operator,  in $\M$.
\end{theo}

\begin{theo} \label{constant}
Let $\s$ be a connected,  non-degenerate hypersurface of $\M$ which is Hopf, i.e.\ its structure vector $\xi$ is a principal direction of $\s$. Then the corresponding principal curvature $a$, i.e.\ defined by  $A\xi = a\xi$, is constant.
\end{theo}

\begin{theo} \label{tubes}

Let $\s$ be a connected,  non-degenerate hypersurface of $\M$  with unit normal $N$.
Assume that $\s$ is Hopf and denote by $a$
 the corresponding principal curvature, i.e. $A\xi = a\xi.$
Then if $c \, \eps \<N,N\>=1$, or if $c \, \eps \<N,N\>=-1$ and $|a|>2$, then $\s$ is, locally, a tube over a  complex or para-complex submanifold.
\end{theo}

\begin{rema} \em In the case $c=1, \eps=1$ and $p=0$, $\M$ is the complex projective space $\C\P^n$, and if $c=-1, \eps=1$ and $p=n$, we have $\M=\C\H^n$,  the complex hyperbolic space. Hence  Theorem \ref{tubes} generalizes \cite{CR} and \cite{Mo}. Observe that in these two cases, the metric being positive, we have $\<N,N\>=1$.
\em
\end{rema}

This paper is organized as follows: in Section \ref{s1}   the geometry of the pseudo-Riemannian complex and the para-complex space forms is described. Section \ref{s2} contains basic relations about the geometry of real hypersurfaces in $\M$ and the proof of Theorem \ref{one}. In Section \ref{s3}  four Lemmas about real hypersurfaces and the proof of Theorem \ref{constant} are presented. Finally, in Section \ref{s4} the proof of Theorem \ref{tubes} is given and at the end of the Section some open problems are proposed for further research on this area.

\section{The ambient spaces: pseudo-Riemannian complex and para-complex space forms} \label{s1}
\subsection{The abstract structures}
All along the paper the ambient space will be a $2n$-dimensional pseudo-Riemannian manifold $(\M,\<\cdot,\cdot\>, J)$ endowed with is a complex or para-complex structure $J$, i.e.\ a $(1,1)$ tensor field satisfying $J^2 = -\eps Id$ which is compatible with  respect to $\<\cdot,\cdot\>$, i.e.\
$$ \< J \cdot,J\cdot\> = \eps \<\cdot,\cdot\>.$$
In other words, $J$ is an isometry in the complex case and an anti-isometry in the para-complex case.
This assumption implies that the signature of $\<\cdot,\cdot\>$ must be even in the complex case and neutral in the para-complex case.

The bilinear map $\omega(X,Y) := \<JX,Y\>$ is alternate and non-degenerate. Furthermore, the 2-form $\omega$ is closed, hence symplectic. Therefore, the triple $(\<\cdot,\cdot\>,J, \omega)$ is a \em pseudo-K\"ahler \em or \em para-K\"ahler \em structure.

We assume furthermore that the curvature ${R}$ of $\<\cdot,\cdot\>$ satisfies
$${R}(X,Y) ={c} \left(  \eps  X \wedge Y +  J X \wedge JY + 2 \< X,JY\>J \right),$$
where the notation $X\wedge Y$ denotes the operator $Z \to (X \wedge Y) Z= \<Y,Z\>X -\<X,Z\>Y$ and where $c$ is a real constant. Observe that  if  $X$ is a non-null vector, we have $\<R(X,JX)JX,X\>=4c, $ i.e.\ any complex or para-complex $2$-plane $Span(X,JX)$ has sectional curvature $4c$. The constant  $4c$ is called the \em holomorphic \em  or \em para-holomorphic \em curvature of
$(\M , \<\cdot,\cdot\>,J)$.

Observe that the rescaled  $\la \<\cdot,\cdot\>$, where $\la$ is a positive constant has holomorphic curvature $\la^{-2} c$. On the other hand,
replacing the metric $\<\cdot,\cdot\>$
 by its opposite $-\<\cdot,\cdot\>$  leaves invariant the curvature operator ${R}$. It follows that if $(\M,\<\cdot,\cdot\>,J)$ has (para-)holomorphic curvature $4c$, then  $(\M,-\<\cdot,\cdot\>,J)$ has (para-)holomorphic curvature $-4c$.

In the next two sections  instances of such manifolds will be described explicitly.

\subsection{Pseudo-Riemannian complex space forms}

We consider the space $\C^{n+1}$ endowed with the pseudo-Hermitian form:
$$ \<\<\cdot,\cdot\>\>_p = - \sum_{j=1}^p dz_j d\bar{z}_j +   \sum_{j=p+1}^{n+1} dz_j d\bar{z}_j $$
The corresponding metric $\<\cdot,\cdot\>_{2p} := \Re\<\<\cdot,\cdot\>\>_p$  has signature $(2p, 2(n+1-p)).$
We define the hyperquadrics
$$ \S^{2n+1}_{p,c}:=\{ z \in \C^{n+1} | \,  \<z,z\>_{2p}  =c \},$$
For example,  $\S^{2n+1}_{0,1} = \S^{2n+1}$ is the round unit sphere;

The pseudo-Riemannian complex space forms are the quotients of these  hyperquadrics by the natural $\S^1$-action:
$$ \C \P^n_{p,c} := \S^{2n+1}_{2p,c}   \slash \sim,$$
where $ z \sim z' $ if there exists $\te \in \R$ such that $z' = (\cos \te, \sin \te) .z$.
In particular
\begin{itemize}
\item[-] $\C\P^{n}_{0,1} = \C\P^n$ is the complex projective space;
\item[-] $\C\P^n_{n,-1} =\C \H^n$ is the complex hyperbolic space;
\end{itemize}
We denote by $\pi$ the canonical projection $\pi: \S^{2n+1}_{2p,c} \to \C\P^n_{p,c}. $
We endow $\C\P^n_p$ with the metric $\<\cdot,\cdot\>$ that makes the projection $\pi$ a pseudo-Riemannian submersion.
The  projection $\pi: \S^{2n+1}_{2p,c} \to \C\P^n_{p,c}$ also induces a natural complex structure in $\C\P^n_{p,c}$.
It is easy to check that $(\C\P^n_{p, c},J,\<\cdot,\cdot\>)$ is a pseudo-K\"{a}hler manifold and that its curvature tensor satisfies
$$ R(X,Y) ={c} \big(  X \wedge Y + J X \wedge JY + 2 \< X,JY\>J \big).$$
In particular, $\C\P^n_{p,c}$ has constant holomorphic curvature $4c$.

Observe that the involutive map $ (z_1, \ldots , z_n ) \mapsto (z_{p+1}, \ldots , z_n, z_1, \ldots z_p)$ is an anti-isometry between $\S^{2n+1}_{p, c} $   and $\S^{2n+1}_{n+1-p, -c}$. It follows that the spaces $\C\P^n_{p,c}$ and $\C\P^n_{n+1-p, -c}$ are anti-isometric.
\subsection{Para-complex space forms}
The set of \em  para-complex \em (or \em split-complex, \em or \em double\em)  numbers $\D$
is the two-dimensional real vector space $\R^2$ endowed with the commutative algebra structure whose product  rule is given by
$$ (x , y) . (x' , y')= (xx'+yy', xy'+x'y).$$
 The para-complex projective plane is the set
of para-complex lines of $\D^{n+1}$.
We consider the neutral metric
$$\<\cdot,\cdot\>_*:= \sum_{j=1}^n dx_j^2- dy_j^2  $$
and the hyperquadric
$$ \S^{2n+1}_{n+1,-1}:= \{ z \in \D^{n+1} | \,  \<z,z\>_* =-1 \} .$$
Then we define:
$$ \D\P^n := \S^{2n+1}_{n+1,-1} \slash \sim,$$
where $ z \sim z' $ if there exists $\te \in \R$ such that $z' = (\cosh \te, \sinh \te) .z$.
We endow $\D\P^n$ with the metric $g$ that makes the projection $\pi: \S^{2n+1}_{n,-1} \to \D\P^n$ a pseudo-Riemannian submersion. The metric $g$ has neutral signature $(n,n)$. For technical reasons it is convenient to introduce the "polar" space $\overline{\D\P^n}$ of $\D\P^n$ by
$$ \overline{\D\P^n} := \S^{2n+1}_{n+1,1} \slash \sim.$$
The  anti-isometry $J$ of $\D^{n+1}$ induces  canonically an anti-isometry between $\D\P^n$ and $ \overline{\D\P^n}$.

\medskip

According to \cite{GM}, the curvature operator of $\D\P^n$ is given by
$$ R(X,Y) =-X \wedge Y +  J X \wedge JY + 2 \< X,JY\>J   .$$
In particular,  $\D\P^n$ has constant para-holomorphic curvature $4$ (but it is not characterized by this property).
On the other hand, $\overline{\D\P^n}$ has constant para-holomorphic curvature $-4$.




\section{Auxiliary relations about real hypersurfaces and proof of Theorem \ref{one} } \label{s2}
In this section let $\s$ be an immersed real hypersurface in $\M$, whose induced metric is non-degenerate.  This
 implies the local existence of a unit normal vector field $N$. After a possible change of metric $\<\cdot,\cdot\> \to -\<\cdot,\cdot\>,$ there is no loss of generality in assuming that $\<N,N\>=1$, and we will  do so in the remainder of the paper. Observe that reversing the metric has the effecting of reversing its curvature $c$. Hence, without loss of generality, we could alternatively assume that $c=1$ and let $\<N,N\>$ take the two possible values $\pm 1$. However the first choice seems more natural.

\subsection{The structure of a real hypersurface in \M}
The \emph{structure vector field }$\xi$ is given  by
\begin{eqnarray}\label{a1}
\xi:=-\eps JN.
\end{eqnarray}
It follows that $N=J\xi$ and that $\<\xi,\xi\>=\eps $. The orthogonal complement $\xi^\perp:= \Ho$,  a $(2n-2)$-dimensional subspace of $T \s$, will be refered as to the \em horizontal distribution. \em Given a vector $X$ tangent to $\s$, the vector $JX$ is not necessarily tangent to $\s$ but its tangential part, that we denote by $\varphi$, is horizontal. Introducing  the one-form $\eta:=\<J \cdot,N\>$,  we have
\begin{eqnarray}\label{a3}
JX&=&\varphi X+  \<JX,N \> N \\
&=&\varphi X + \eta(X)N. \nonumber
\end{eqnarray}
Observe also that
$$\eta(X)=\<JX,N\> =-\< X,JN\> =\eps\<X,\xi\>$$
and
\begin{eqnarray}\label{a2}
\eta(\xi)=\<J\xi,N\>=\<N,N\>=1.
\end{eqnarray}
On the other hand, doing $X=\xi$ in Equation (\ref{a3}), we get
\begin{eqnarray}\label{a4}
\varphi\xi=0.
\end{eqnarray}

Now, we have
\begin{eqnarray}
-\eps X&=&J^{2}X   \nonumber  \\
&=&J \big(\varphi X+ \eta(X)N \big)\nonumber\\
&=&J(\varphi X) + \eta(X) JN \nonumber\\
&=&\varphi (\varphi X) + \eta(\varphi X)N+ \eta(X)JN.
\end{eqnarray}
Considering the tangent and normal parts of this equation, we get that  $\eta \circ \varphi =0$ and
$$ -\eps X = \varphi^2 X -\eps  \eta(X)\xi,$$
so that
\begin{eqnarray}\label{a5}
\varphi^2 = -\eps Id +\eps  \eta ( \cdot) \xi.
\end{eqnarray}
Finally, denoting by $g$ the induced  metric  on $\s$, we have the following relation:
\begin{eqnarray}\label{a6}
\<\varphi X, \varphi Y\>&=&\<JX- \eta(X)N,JY- \eta(Y)N\> \nonumber\\
&=&\< JX,JY\>-\eta(X) \<N,JY\>- \eta(Y)\<N,JX\>+\eta(X)\eta(Y) \<N,N\>\nonumber\\
&=&\eps \<X,Y\>- 2 \eta(X)\eta(Y) +\eta(X)\eta(Y)  \nonumber\\
&=&\eps \<X,Y\>-  \eta(X)\eta(Y).
\end{eqnarray}

We conclude that, according to
Relations (\ref{a1}), (\ref{a2}), (\ref{a4}), (\ref{a5}) and (\ref{a6}), the quadruple  $(\varphi,\eta,\xi,\<\cdot,\cdot\>)$
defines  an \emph{almost contact metric structure} on $\s$ when $\eps=1$ and  an \emph{almost para-contact metric structure} on $\s$  when $\eps=-1$.

The Gauss and the Weingarten formulas are respectively given by the equations
\begin{eqnarray}
{\nabla}_{X}Y&=&\overline{\nabla}_{X}Y+ \<AX,Y\>N  \label{Gauss-formula}\\
{\nabla}_{X}N&=&-AX,\label{Weingarten-formula}
\end{eqnarray}
where $\nabla$ and  $\overline{\nabla}$ are the Levi-Civita connection on $\M$ and $\s$ respectively and $A$ is the shape operator of $\s$ with respect to $N$.
Denoting by $\overline{R}$ and $R$ the curvature of $\overline{\nabla}$ and $\nabla$ respectively, the
Gauss equation takes the form:
$$ \<{R}(X,Y)Z, W\> = \<\overline{R}(X,Y)Z,W\> + \< AX, Z\> \<AY, W\>-\< AX, W\>\<AY,Z\>$$
for $X,Y,Z$ and $W$ tangent to $\s$.
Hence
\begin{eqnarray*}\overline{R}(X,Y)Z &=&c\Big( \eps  X \wedge Y +    J X \wedge JY + 2 \< X,JY\> \Big)Z + \<AY,Z\> AX - \<AX,Z\> AY\\
      &=&c \Big( \eps  X \wedge Y +  \varphi X \wedge \varphi Y + 2 \< X, \varphi Y\>\varphi \Big) Z + (AX \wedge AY)Z,
\end{eqnarray*}
so that
$$ \overline{R}(X,Y)= AX \wedge AY + c\left(  \eps X \wedge Y +  \varphi X \wedge \varphi Y + 2 \< X, \varphi Y\> \varphi \right).$$
\bigskip
We now deal with Codazzi equation: for $X,Y$ and $Z$ tangent to $\s$, we have
$$ \< R(X,Y)Z, N\>= \< (\overline{\nabla}_X A)Y -(\overline{\nabla}_Y A)X, Z  \>.$$
Using the expression of ${R}$, we have
\begin{eqnarray*} \<R(X,Y)Z, N\>&=&c\Big( \eps  \<(X \wedge Y)Z,N\> +\<( J X \wedge JY )Z,N\>+
 2 \< X,JY\>\<JZ,N\> \Big) \\
 &=&c\Big(  \eps \<Y,Z\>\<X,N\> - \eps \<X,Z\> \< Y, N\>+ \<JY,Z\>\<JX,N\> - \<JX,Z\> \< JY, N\>\\
  & &  - 2\<X, \varphi Y\> \<Z, JN\> \Big)\\
&=&c \ \Big(  \<\varphi Y,Z\>\eta(X)- \<\varphi X,Z\>\eta(Y) + 2\eps \<X, \varphi Y\> \<Z, \xi\> \Big)
\end{eqnarray*}
to get
\begin{eqnarray}\label{Codazzi-equation}
(\overline{\nabla}_X A)Y - (\overline{\nabla}_Y A)X = c  \Big( \eta(X) \varphi Y - \eta(Y) \varphi X +
 2  \eps \< X, \varphi Y\> \xi \Big).
 \end{eqnarray}


\subsection{Proof of Theorem \ref{one}}

The proof is an easy consequence of the Codazzi equation.

\medskip

\noindent Assume first that $\s$ is umbilic, i.e.\  there exists $\lambda \in C^{\infty}(\s)$ such that $A=\lambda Id$. Then the Codazzi equation (\ref{Codazzi-equation}) becomes:
\[(X\cdot\lambda)Y-(Y\cdot\lambda)X=c \Big( \eta(X)\varphi Y- \eta(Y)\varphi X +2\eps \<X,\varphi Y\> \xi \Big),\]
Taking $X$ horizontal and non-vanishing, and $Y=\xi$ yields
\[(X\cdot\lambda)\xi-(\xi\cdot\lambda)X=c  \varphi X.\]
The inner product of the above relation with $\varphi X$ implies $c=0$, which is a contradiction.

\bigskip

 Assume now that $\s$ has parallel shape operator, i.e. $(\overline{\nabla}_{X}A)Y=0$, for any tangent vectors $X,Y$. Then  the Codazzi equation becomes
$$ 0 = c \Big( \eta(X)\varphi Y-\eta(Y)\varphi X +2 \eps\<X,\varphi Y\> \xi \Big).$$
Taking $X$ horizontal and non-vanishing, and $Y=\xi$ yields
$c  \varphi X=0.$
Since $\varphi X$ does not vanish, we get $c=0$, a contradiction.

\section{Proof of Theorem \ref{constant}}  \label{s3}

Before providing the proof of Theorem some basic Lemmas which hold for real hypersurfaces in $\M$ are given.
\subsection{Basic Lemmas}
\begin{lemm} \label{Prop13} \label{l1}
Let \s be a real hypersurface in $\M$. Then:
\begin{eqnarray}\label{b1}
\overline{\nabla}_X \xi = \eps\varphi A X
\end{eqnarray}
and
\begin{eqnarray}\label{b2}
(\overline{\nabla}_X \varphi)Y =  \eta(Y) AX-\eps\< AX,Y \>\xi .
\end{eqnarray}
\end{lemm}

\begin{proof}
Using successively Gauss equation,  Weingarten equation  and Equation (\ref{a3}), we first calculate
\begin{eqnarray*}
\overline{\nabla}_X \xi &=&{\nabla}_{X} \xi -  \< AX, \xi\> N\\
 &=&-\eps {\nabla}_X JN  -  \< AX, \xi\> N\\
&=& -\eps J  \nabla_X N -  \< AX, \xi\> N \\
&=&  \eps J A X-  \<AX,\xi\>N\\
&=&\eps \varphi AX +  \eta(AX) -  \<AX,\xi\>N.
\end{eqnarray*}
Taking the tangential part of this  yields Equation (\ref{b1}).

As for Equation (\ref{b2}), using again the Gauss, Weingarten equation and  Equation (\ref{a3}),  we have
\begin{eqnarray*} (\overline{\nabla}_X \varphi ) Y&=& \overline{\nabla}_X (\varphi Y) -\varphi (\overline{\nabla}_X Y)\\
&=& \nabla_X (\varphi Y) -   \< AX, \varphi Y\>N  -\varphi (\overline{\nabla}_X Y)\\
&=&\nabla_X \big(JY -\eps \eta(Y)N \big)  -   \< AX, \varphi Y\>N  -\varphi (\overline{\nabla}_X Y)\\
&=&J \nabla_X Y -\eps \nabla_X (\eta(Y)N )  -   \< AX, \varphi Y\>N  -\varphi (\overline{\nabla}_X Y)\\
&=&J \big( \overline{\nabla}_X Y   + \<AX,Y\>N  \big)
 -\eps  \big( \<\nabla_X Y, \xi\>N +\<Y,\nabla_X \xi\> N + \<Y,\xi\>\nabla_X N \big)  \\
&&  \quad \quad -   \< AX, \varphi Y\>N  -\varphi (\overline{\nabla}_X Y)\\
&=&  \eta(\overline{\nabla}_X Y) N +  \<AX,Y\> JN
- \eps  \big( \<\overline{\nabla}_X Y, \xi\>N +\<Y,\overline{\nabla}_X \xi\> N - \<Y,\xi\>AX \big)  \\
&&  \quad \quad -   \< AX, \varphi Y\>N \\
&=& -\eps  \<AX,Y\> \xi - \<Y, \varphi AX\> N + \eps  \<Y,\xi\> AX +  \<\varphi AX,Y\> N\\
&=& \big( \eta(Y) AX - \eps \<AX,Y\> \xi\big).
\end{eqnarray*} \end{proof}

\begin{lemm}  \label{l2}
The following two relations hold on a hypersurface $\s$ of $\M$:
\begin{eqnarray}
\<(\overline{\nabla}_X A)Y-(\overline{\nabla}_Y A)X,\xi \>=2c  \< X,\varphi Y \>,\label{b3}\\
\<(\overline{\nabla}_{X}A)\xi, \xi\>=\<(\overline{\nabla}_{\xi}A)X, \xi\>=\<(\overline{\nabla}_{\xi}A)\xi,X\>.\label{b4}
\end{eqnarray}
\end{lemm}
\begin{proof}
Taking the inner product of Codazzi equation (\ref{Codazzi-equation}) with $\xi$, recalling that $\<\xi,\xi\>=~\eps $, implies  (\ref{b3}).

The first equality in  (\ref{b4}) is  a particular case of (\ref{b3}) making $Y=\xi$.
For the second equality in (\ref{b4}) we have
\begin{eqnarray*}
\<(\overline{\nabla}_{\xi}A)\xi,X\>&=&\<\overline{\nabla}_\xi(A\xi), X\>-\<A\overline{\nabla}_{\xi}\xi,X\>\\
&=&\xi\cdot \<A\xi,X\>-\<A\xi,\overline{\nabla}_{\xi}X\>-\<\overline{\nabla}_{\xi}\xi,AX\>\\
&=&\xi\cdot \<A\xi,X\>-\<A\xi,\overline{\nabla}_\xi X\>-\xi \cdot \<\xi, AX\> +\<\xi,\overline{\nabla}_\xi (AX)\>\\
&=&-\<\xi,A\overline{\nabla}_\xi X\>+\<\xi,\overline{\nabla}_\xi (AX)\>\\
&=&\<\xi, (\overline{\nabla}_\xi A)X\>.
\end{eqnarray*}
\end{proof}

\begin{lemm}  \label{l3}
 Let $\s$ be a Hopf  hypersurface in $\M$ and $a$ the Hopf curvature, i.e. $A\xi = a\xi.$ Then the following relations hold on \s
\begin{eqnarray}
&&grad\; a=\eps  (\xi\cdot a)\xi,   \label{b5}\\
&&A\varphi A-\frac{a}{2}(A\varphi+\varphi A)-c \eps  \varphi=0,   \label{b6}\\
&&(\xi\cdot a)(\varphi A+A\varphi)=0   \label{b7}.
\end{eqnarray}
\end{lemm}

\begin{proof}

--- Proof of (\ref{b5}):
we first calculate, using several times Equation (\ref{b1}),
\begin{eqnarray*}
 (\overline{\nabla}_X A)\xi &=&  \overline{\nabla}_X (A\xi) - A \overline{\nabla}_X \xi \\
 &=& (X \cdot a )  \xi  + a \overline{\nabla}_X \xi -\eps A \varphi A X\\
&=&   (X \cdot a )  \xi  + \eps a\varphi AX - \eps A \varphi A X.
\end{eqnarray*}
Hence we obtain
\begin{eqnarray} \label{c1}
 (\overline{\nabla}_X A)\xi &=&  (X \cdot a )  \xi  +  \eps  (a Id - A) \varphi A X.
\end{eqnarray}
Taking the inner product of (\ref{c1}) with $\xi$ yields (taking into account that $\<\xi, \xi\> =~\!\eps  $)
\begin{eqnarray} \<(\overline{\nabla}_X A)\xi, \xi \> = \eps   (X \cdot a )=\eps  \< grad \;  a , X\>. \label{c2}\end{eqnarray}
On the other hand, making $X=\xi$ and recalling that $\varphi \xi$ vanishes, we get
\begin{eqnarray*}  (\overline{\nabla}_\xi A )\xi =   (\xi \cdot a )  \xi.\end{eqnarray*}
Putting together these last two equations, we conclude, using Lemma \ref{l2},
\begin{eqnarray*} \< grad \;  a , X\>&=& \eps  \<(\overline{\nabla}_X A)\xi, \xi \>\\
&=&  \eps  \<(\overline{\nabla}_\xi A)\xi, X \>\\
&=&  \eps  (\xi \cdot a ) \<\xi,X\>,
\end{eqnarray*}
from which Equation (\ref{b5}) follows.

\medskip

\noindent  --- Proof of (\ref{b6}):
first,  by an easy calculation,
$$ \<(\overline{\nabla}_X A)Y,\xi\> =\<(\overline{\nabla}_X A)\xi,Y\>$$
Then, using Equations (\ref{c1}), (\ref{c2}) and Lemma \ref{l2}, we get
\begin{eqnarray*}
\<(\overline{\nabla}_X A)Y,\xi\> &=&\<(\overline{\nabla}_X A)\xi,Y\>\\
&=& (X \cdot a)\<\xi,Y\>  + \eps \< (aId - A)\varphi AX,Y\>\\
&=&\eps  \<(\overline{\nabla}_ \xi A)\xi, X\>    \<\xi,Y\>  +\eps  \< (a Id - A)\varphi AX,Y\>\\
&=&  \eps (\xi \cdot a) \<\xi,X\>  \<\xi,Y\>  +\eps \< (a Id - A)\varphi AX,Y\>.
\end{eqnarray*}
Interchanging $X$ and $Y$ and substracting, we calculate
\begin{eqnarray*}
\<(\overline{\nabla}_X A)Y-(\overline{\nabla}_Y A)X,\xi\>&=&
\eps \big( \< (a Id - A)\varphi AX,Y\>-  \< (a Id - A)\varphi AY,X\>\big).
\end{eqnarray*}
Now, from  (\ref{b3}) (Lemma \ref{l2}) this implies
\begin{eqnarray*}
\eps \big( \< (a Id - A)\varphi AX,Y\>-  \< (a Id - A)\varphi AY,X\>\big)&=&2 c \eps \< X,\varphi Y\> \<\xi,\xi\>\\
&=& 2    c \< X,\varphi Y\>.
\end{eqnarray*}
It follows, using the facts that $A$ is self-adjoint (and therefore $aId -A$ as well) and that $\varphi$ is skew-symmetric, that
\begin{eqnarray*}
2\eps    c \< X,\varphi Y\>&=& \< (a Id - A)\varphi AX,Y\>-  \< (a Id - A)\varphi AY,X\>\\
&=& -\< X,   A \varphi (aId -A) Y\> -  \<X, (a Id - A)\varphi AY\>,
\end{eqnarray*}
which implies that
\begin{eqnarray*} 2\eps    c \varphi &=&- A \varphi (aId -A) -  (a Id - A)\varphi A\\
 &=& -a(A\varphi + \varphi A) +2 A\varphi A, \end{eqnarray*}
from which Equation (\ref{b6}) follows.

\bigskip

\noindent  --- Proof of (\ref{b7}): setting $\be:= \eps \xi \cdot a$ (so in particular $\mbox{grad}\;a = \be \xi$), we have
\begin{eqnarray*}
\<\overline{\nabla}_X (\mbox{grad}\;a) , Y \> -\<\overline{\nabla}_Y (\mbox{grad}\;a) , X \> &=& X  \cdot\<\mbox{grad}\;a , Y \> - \<\mbox{grad}\;a  , \overline{\nabla}_X Y \>   \\
 &&\; \; \; \quad -Y  \cdot \<\mbox{grad}\;a , X \> + \<\mbox{grad}\;a  , \overline{\nabla}_Y X \>  \\
&=&X \cdot (Y \cdot a ) - Y \cdot (X \cdot a ) + \<\mbox{grad}\;a  , \overline{\nabla}_X Y - \overline{\nabla}_Y X \>\\
&=&  \left([X,Y]-\overline{\nabla}_X Y - \overline{\nabla}_Y X \> \right) \cdot a \\
&=&0.
\end{eqnarray*}
It follows that
\begin{eqnarray} \label{c3}
0&=&\<\overline{\nabla}_X  \be \xi , Y \> -\<\overline{\nabla}_Y \be \xi , X \>  \nonumber \\
&=& (X \cdot \be) \< \xi,Y\> -  \be \< \varphi A X,Y\> -(Y \cdot \be) \< \xi,X\> -  \be \< \varphi A Y,X\>  \nonumber \\
&=& (X \cdot \be) \< \xi,Y\> -(Y \cdot \be) \< \xi,X\> -  \be \<(\varphi A +A \varphi  )X,Y\>.
\end{eqnarray}
Making $Y=\xi$ yields
\begin{eqnarray*}
0&=& X \cdot \be \<\xi,\xi\>-(\xi \cdot \be) \< \xi,X\>-  \be \<(\varphi A +A \varphi  )X,\xi\> \\
 &=&\eps  X \cdot \be -(\xi \cdot \be) \< \xi,X\>-   \be \<A \varphi X,\xi\>\\
 &=&\eps  X \cdot \be - (\xi \cdot \be) \< \xi,X\>-   \be \< \varphi X,a\xi\>\\
 &=&\eps  X \cdot \be - (\xi \cdot \be) \< \xi,X\>.
\end{eqnarray*}
Hence we have $X \cdot \be =\eps   (\xi \cdot \be) \< \xi,X\>$, which implies that
$ (X \cdot \be) \< \xi,Y\> =(Y \cdot \be) \< \xi,X\> .$ Combining with (\ref{c3}) yields
$$  \be \, \<(\varphi A +A \varphi  )X,Y\> ,$$
which implies the desired identity.
\end{proof}

\begin{lemm} \label{l4}
If $X$ is an principal vector of $A$ with principal curvature $\ka$, then $\varphi X$  is a principal vector with principal curvature
$$\bar{\ka}:=\frac{ \ka a+2 c \eps  }{2\ka -a}.$$
In particular the principal subspace
 $E_\ka:=\{ X \in T\s |  \, AX =\ka X\}$ is $\varphi$-invariant if and only if $\ka^2 - a\ka - c \eps   =0$.
\end{lemm}

\begin{proof}
We write Equation  (\ref{b6}) in the case when $A X=\ka X$:
$$\ka A\varphi X-\frac{a}{2}(A\varphi X+ \ka \varphi X)-c \eps  \varphi X=0,$$
so that
$$ (\ka -\frac{a}{2}) A\varphi X = (c \eps + \frac{a \ka}{2}) \varphi X,$$
i.e.
$$ A \varphi X= \frac{ \ka a+2  c\eps  }{2\ka -a} \varphi X,$$
so we get the required expression for $\bar{\ka}$ satisfying $A (\varphi X) = \bar{\ka} (\varphi X)$. Finally, is $E_\ka$ is $\varphi$-stable, we must have $\ka= \frac{ \ka a+2 c \eps  }{2\ka -a}$, which implies the last claim of the Lemma.
\end{proof}


\noindent \textit{Proof of Theorem  \ref{constant}.}

\smallskip

 We proceed by contradiction. By Equation (\ref{b5}) if $a$ is not constant, then $\xi\cdot a \neq0$.
Consider $Z$  a horizontal vector. So $\eta(Z)=0$ and, by Lemma \ref{l1}, we have $(\overline{\nabla}_\xi A) Z=0$.  It follows that
\begin{eqnarray}\label{r1}
0 &=& \overline{\nabla}_{\xi} \big( (\varphi A+A\varphi)Z \big) \nonumber \\
&=&
\varphi(\overline{\nabla}_{\xi}A)Z+(\overline{\nabla}_{\xi}A)\varphi Z.
\end{eqnarray}
We now write the Codazzi equation (\ref{Codazzi-equation}) with $X=\xi$ and $Y=Z$, yields, using (\ref{c1}):
\begin{eqnarray*}
(\overline{\nabla}_{\xi}A)Z&=&(\overline{\nabla}_{Z}A)\xi +  c \Big(  \eta(\xi) \varphi Z -  \eta(Z) \varphi \xi +
 2\eps  \< \xi, \varphi Z\> \xi \Big)\\
&=&  (\overline{\nabla}_{Z}A)\xi + c    \varphi Z \\
&=&  (Z\cdot a)\xi+ \eps( aId -A)\varphi AZ+c     \varphi Z\\
&=& \eps( aId -A)\varphi AZ+c    \varphi Z.
\end{eqnarray*}
(Equation (\ref{b5}) implies that $ (Z\cdot a)$ vanishes).
It follows that
\begin{eqnarray*}  \varphi(\overline{\nabla}_{\xi}A)Z&=&  \eps \varphi ( aId -A)\varphi AZ+ c   \varphi^2 Z\\
&=&\eps a \varphi^2 AZ -\eps  \varphi A\varphi A Z  - c \eps   Z\\
&=& -a AZ -  A^2 Z  -c \eps   Z.
 \end{eqnarray*}
Analogously, the Codazzi equation with   $X=\xi$ and $Y=\varphi Z$
\begin{eqnarray*}
(\overline{\nabla}_{\xi}A)\varphi Z&=&(\overline{\nabla}_{\varphi Z}A)\xi +  c \Big( \eta(\xi) \varphi^2 Z - \eta(\varphi Z) \varphi \xi +
 2 \eps\< \xi, \varphi^2 Z\> \xi \Big)\\
&=&  (\overline{\nabla}_{\varphi Z}A)\xi - c\eps     Z \\
&=& (\varphi Z\cdot a)\xi+ \eps( aId -A)\varphi A \varphi Z -  c \eps    Z\\
&=&\eps( aId -A)\varphi A \varphi Z -  c \eps   Z\\
&=&-  a AZ +A^2  Z -  c \eps    Z.
\end{eqnarray*}
By (\ref{r1}) we deduce that
$$  a AZ = -c \eps    Z.$$
This implies $a $ does not vanish, and moreover that the restriction of $A$ to the  horizontal space is $-c \eps   a^{-1} Id.$ Since the horizontal space is $\varphi$-stable, it follows from Lemma \ref{l4}, that
$$ \big(-c \eps   a^{-1}\big)^2  -a \big(- c  \eps   a^{-1}\big) -c \eps  =0,$$
i.e.
$$ \frac{c^2}{a^2}  + c \eps   -c \eps  =0.$$
This implies $c=0$, a contradiction.

\section{Proof of the Theorem \ref{tubes}} \label{s4}

We recall that $\M := \D\P^n$ or $\C \P^n_p$ and that
 $\widetilde{\M}:= \S^{2n+1}_{n,-c}$ or $\S^{2n+1}_{2p, c}$ is a bundle over $\M$ with projection $\pi.$ We still denote by $\<\cdot,\cdot\>$ the metric on $\widetilde{\M}$. We shall denote by $\widetilde{\nabla}$ the Levi-Civita connection on  $\widetilde{\M}$. This is nothing but the tangential part of the flat connection of $\R^{2n+2}$. We denote by $\widetilde{J}$ the complex (resp.\ para-complex) structure of $\C^{n+1}$ (resp.\ $\D^{n+1}$).

Let $F: U \to \M$ be a local parametrization of $\s$ and $\widetilde{F}: \s \to \widetilde{\M}$ a lift of $F$, i.e.\
 $\pi \circ \widetilde{F} = F$. In particular $\<\widetilde{F},\widetilde{F}\>=c\eps$. Observe that the choice of $\widetilde{F}$ is not unique, but none of them
satisfies
 $\< d\widetilde{F}(\cdot),\widetilde{J} \tilde{F}\>=0$, because  the integral submanifolds of the hyperplane distribution $\widetilde{J}^\perp$ have at most dimension~$n$ (Legendrian submanifolds).

By a slight abuse of notation, we still denote by $N$ the composition of the unit normal vector field on $F(U)$ with $F$. In other words,  $N  : U \to T{\M}$.
Let $\widetilde{N}$ be a lift of $N$, i.e.\ $\widetilde{N} : U \to T\widetilde{\M}$ such that
\begin{eqnarray*}
d \pi_{\widetilde{F}(x)} \circ \widetilde{N}(x) &=&N(x), \, \forall x \in U.
\end{eqnarray*}
In particular  $\widetilde{N} \in \widetilde{\M}$ (resp.\   $\overline{\widetilde{\M}}$) if  $\eps=1$ (resp.\ $\eps=-1$) and
$\<\widetilde{N},\widetilde{N}\>=1$. Moreover, we have
$$ \< d\widetilde{F}(\cdot), \widetilde{N}\>=0.$$ Since the immersion $d\widetilde{F}$ has co-dimension 2, the choice of $\widetilde{N}$ is not unique.
Since moreover $d\widetilde{F}$ is tranverse to the vector field $\widetilde{J} \widetilde{F}$, we may choose $\widetilde{N}$ in such a way that
$$\<\widetilde{N}, \widetilde{J} \widetilde{F}\>=0.$$

We denote by $\xi$ the tangent vector field on $U$ such that $\Xi := dF(\xi) = -JN$. We also set
$\widetilde{\Xi}:= -\widetilde{J}\widetilde{N}.$ Of course $\widetilde{\Xi}$ is the lift of $\Xi.$

\smallskip

We now assume that $F(U)$ is Hopf, i.e.\ $A \xi = a \xi$. By Theorem \ref{constant},  $a $ is some real constant. We need some extra notation:
we set $\eps':= c \eps $ and
$$ (\co', \si' ) := \left\{ \begin{array}{ll}  (\cos ,\sin ) & \mbox{ if } \eps'=1,\\
     (\cosh ,\sinh ) & \mbox{ if } \eps'=-1.
\end{array} \right.$$
We also set the obvious notations $\ta':=\frac{\si'}{\co'}$ and $\coe':=\frac{\co'}{\si'}$.
Finally, we
introduce $f:= \pi \circ \widetilde{f} : U \to \M$, where
$$\widetilde{f}:= \co'(\te) \widetilde{F} + \si' (\te) \widetilde{N} $$
and   $\te$ is some real constant to be determined later.

\smallskip

By Proposition 10, p.\ 97 of \cite{An1} (see also \cite{ON}), we have

\begin{eqnarray*} \widetilde{\nabla}_{\widetilde{\Xi}} \widetilde{N}  &=& \widetilde{\nabla_{\Xi} N}  +
\<\widetilde{J} \widetilde{F},\widetilde{J} \widetilde{F}\> \< \widetilde{\Xi}, \widetilde{J} \widetilde{N}\> \widetilde{J} \widetilde{F}\\
&=& \widetilde{\nabla_{\Xi} N}  - c  \widetilde{J} \widetilde{F}\\
&=&- a \widetilde{\Xi} -\eps \eps' \widetilde{J} \widetilde{F}.
\end{eqnarray*}
It follows that
\begin{eqnarray*}  d\widetilde{f}(\xi) &=&\co'(\te) d\widetilde{F}(\xi) + \si'(\te)d \widetilde{N} (\xi)\\
&= &\co'(\te) \widetilde{\Xi}+ \si'(\te)\widetilde{\nabla}_{\widetilde{\Xi}} \widetilde{N}\\
&=&-\eps \co'(\te) \widetilde{J} \widetilde{N}+\si' (\te) \left( a\eps \widetilde{J} \widetilde{N} - \eps \eps' \widetilde{J} \widetilde{F}\right)\\
& = & -\eps  \Big(\eps' \si'(\te) \widetilde{J} \widetilde{F} + (\co' (\te)- a \si'(\te) )\widetilde{J} \widetilde{N}\Big).
\end{eqnarray*}
We now choose $\te$ in order to have $df(\xi)=0$, i.e.\ $d \widetilde{f}(\xi) \in \widetilde{J} \widetilde{f}\, \R$.
This is equivalent to the existence of $\la \in \R$ such that
$$\left\{ \begin{array}{lll} \eps' \si'(\te) &=& \la \co'(\te), \\
 \co'(\te)- a \si'(\te) &=& \la \si'(\te).
\end{array} \right.$$
It follows that $\la =\eps'  \ta' (\te)$ and
$$ a = \frac{\co' (\te)}{\si' (\te)}-\la= \frac{(\co'(\te))^2-\eps' (\si'(\te))^2 }{\co' (\te) \si' (\te)}= 2 \coe'(2 \te)$$
so that, taking $$\te:=\frac{1}{2} \coe'^{-1} (a/2),$$
we get  $df(\xi)=0$.
This is possible for all real number $a$ if $\eps'=1$, and if $|a|>2$ if $\eps'=-1$.\footnote{In the case $\eps'=-1$, if we instead set
$$\widetilde{f}':=\sinh(\te) \widetilde{F} + \cosh (\te) \widetilde{N},$$
which is valued in $\overline{\widetilde{\M}}$, we get again
$a=2 \coth (2\te)$. The map $f':=\pi \circ \widetilde{f}'$ is the \em polar \em of $f := \pi \circ \widetilde{f}.$} Hence $f$ is constant
along the integral lines of $\xi$. In particular, the rank of $f$ is strictly less than $2n-1$.

\bigskip

We now claim that the rank of $f $ is even and that its image is a complex submanifold of $\M$.
We first calculate, for a horizontal vector  $v \in \Ho$:
\begin{eqnarray*} d \widetilde{f}(v) &=&\co'(\te) d\widetilde{F} (v)+ \si'(\te)d \widetilde{N} (v)\\
&=& \co'(\te) d\widetilde{F} (v) + \si'(\te) \widetilde{\nabla}_{ d\widetilde{F} (v)}        \widetilde{N}\\
&=&\co'(\te) d\widetilde{F} (v) + \si'(\te) \Big( -d\widetilde{F} (Av)
+ \< \widetilde{J}   \widetilde{F}, \widetilde{J} \widetilde{F} \> \< d\widetilde{F} (v), \widetilde{J}   \widetilde{N}\> \widetilde{J}\widetilde{F}  \Big)\\
&=& \co'(\te) d\widetilde{F} (v) - \si'(\te) d\widetilde{F} (Av) \\
&=& d\widetilde{F} (\co' (\te)v -\si'(\te)Av).
\end{eqnarray*}
Hence $Ker ( df) = Ker (\co' (\te) Id - \si'(\te) A)$ and therefore
$$ rank (f) = 2n - {\rm dim} Ker (\co'(\te) Id - \si'(\te) A)$$
Moreover, $v \in Ker(df)$ if and only if $Av = \coe'(\te) v$, i.e.\ $\coe'(\te)$ is a principal curvature of $F.$
Observe that $\coe'^2 (\te)-  \frac{\coe' (\te)}{2} a -\eps'=0$, so by Lemma \ref{l4} of Section \ref{s2}, the corresponding eigenspace is $J$-invariant. In particular the rank of $f$ is even.

  If $v$ does not belong to $Ker (\co'(\te) Id - \si'(\te) A),$ we claim that there exists $w $
such that $Jdf(v)= df(w).$
Since
$$  \widetilde{J}d \widetilde{f}(v)= d\widetilde{F}( \co'(\te) \varphi v - \si' (\te) \varphi Av)$$
this  is equivalent to
$$\varphi (\co' (\te) Id - \si' (\te) A) v =( \co'(\te) Id -\si' (\te) A)w.$$
Hence, we get the required relation setting
$$ w := ( \co'(\te) Id -\si' (\te) A)^{-1} \varphi(\co' (\te) Id - \si' (\te) A) v.$$
 This proves that $df(T U)$ is stable with respect to $J$, i.e.\ $f(U)$ is a complex submanifold.
The easy task to check that $F(U)$ is the tube of radius $\te$ over $f(U)$ is left to the reader.

\subsection{Open Problems}
Summarizing, in this paper some basic results  are presented and a  characterization of real hypersurfaces with Hopf curvature satisfying $|\alpha|>2$ in pseudo-Riemannian complex space forms and para-complex space forms is given. Therefore, a first question which is raised in a natural way is:

\emph{Are there real hypersurfaces in pseudo-Riemannian complex space forms or para-complex space forms whose Hopf curvature is small,  i.e. $|\alpha|\leq2$?}

Following similar steps to those which have been done in the study of real hypersurfaces in the cases of complex space forms, complex two-plane Grassmannians, etc., a great amount of questions concerning real hypersurfaces in pseudo-Riemannian complex space forms and para-complex space forms come up. For instance, it would be interesting to answer the following:

\emph{Are there real hypersurfaces in pseudo-Riemannian complex space forms or para-complex space forms whose shape operator commutes with $\varphi$, i.e. $A\varphi=\varphi A$?}


\begin{multicols}{2}
\noindent   Henri Anciaux \\
Universidade de S\~ao Paulo, IME \\
  1010 Rua do Mat\~ao,  \\
Cidade Universit\'aria   \\
 05508-090 S\~ao Paulo, Brazil   \\
henri.anciaux@gmail.com \\

\columnbreak

\noindent Konstantina Panagiotidou \\
Faculty of Engineering \\
Aristotle University of Thessaloniki \\
Thessaloniki 54124, Greece \\
kapanagi@gen.auth.gr

\end{multicols}

\begin{thebibliography}{XXXXX}

\bibitem[An1]{An1} H. Anciaux, \em Minimal submanifolds in Pseudo-Riemannian geometry, \em World Scientific, 2010

\bibitem[An2]{An2} H. Anciaux, \em Surfaces with one constant principal curvature in three-dimensional space forms, \em arXiv:1307.6735


\bibitem[BD]{BD} A.\ Bejancu, K.\ L.\ Duggal, \emph{Real hypersurfaces of indefinite Kaehler manifolds}, Internat.\ J.\ Math.\  and Math.\ Sci.\, \textbf{16} no. 3, (1993),  545--556


\bibitem[CR]{CR} T.\ Cecil, P.\ Ryan, \em Focal sets and real hypersurfaces in complex projective space, \em Trans.\ Amer. Math. Soc. {\bf 269}  (1982),
 481--499

\bibitem[GM]{GM} P.\ M.\ Gadea, A. M. Montesinos Amilibia, \em Spaces of constant para-holomorphic curvature, \em
Pacific J.\ of Maths.\ {\bf 136} no. 1, (1989), 85--101



\bibitem[IR]{IR} T.\ Ivey, P.\ Ryan, \em Hopf Hypersurfaces of Small Hopf Principal Curvature in $\C\H^2$, \em
 Geom. Dedicata {\bf 141} (2009), 147--161

\bibitem[Iv]{Iv} T.\ Ivey, \em A d'Alembert Formula for Hopf Hypersurfaces, \em
 Results in Maths. {\bf 60} (2011), 293--309


\bibitem[KS]{KS}  U.-H.\ Ki, Y.-J.\ Suh, \em On real hypersurfaces of a complex space form, \em  Math. J.\ Okayama Univ. {\bf 32}  (1990), 207--221



\bibitem[Ki]{Ki} M.\ Kimura, \em Hopf hypersurfaces in nonflat complex space forms, \em Proceedings of The Sixteenth International Workshop on Diff. Geom. \textbf{16} (2012) 25--34

\bibitem[NR]{NR}  R.\ Niebergall, P.\ Ryan, \em Real Hypersurfaces
in Complex Space Forms, \em Tight and Taut Submanifolds MSRI Publications
Volume 32, 1997

\bibitem[Ma]{Ma} Y.\ Madea, \em On real hypersurfaces of a complex projective space, \em  J. Math. Soc. Japan {\bf 28} (1976), 529--540


\bibitem[Mo]{Mo}  S.\ Montiel,  \em Real hypersurfaces of a complex hyperbolic space, \em J. Math. Soc. Japan \textbf{37} (1985), no. 3, 515--535

\bibitem[ON]{ON} O'Neill, \em The fundamental equations of a submersion, \em Michigan Math. J., \textbf{13} (1966), 459--469

\end{thebibliography}
\end{document}